\documentclass[11pt,oneside,english]{amsart}
\usepackage[T1]{fontenc}
\usepackage{geometry}
\usepackage{amstext}
\usepackage{amsthm}
\usepackage{setspace}
\usepackage{esint}
\makeatletter
\theoremstyle{plain}
\newtheorem{thm}{\protect\theoremname}
  \theoremstyle{plain}
  \newtheorem{lem}[thm]{\protect\lemmaname}
  \theoremstyle{plain}
  \newtheorem{prop}[thm]{\protect\propositionname}
  \theoremstyle{plain}
  \newtheorem{cor}[thm]{\protect\corollaryname}

\makeatother

\usepackage{babel}
  \providecommand{\corollaryname}{Corollary}
  \providecommand{\lemmaname}{Lemma}
  \providecommand{\propositionname}{Proposition}
\providecommand{\theoremname}{Theorem}

\begin{document}

\title{A Note on Jointly Modeling Edges and Node Attributes of a Network}

\author{Haiyan Cai}


\address{Department of Mathematics and Computer Science, University of Missouri
- St. Louis, St. Louis, MO 63121.}

\email{haiyan\_cai@umsl.edu}

\subjclass[2000]{60, 62, 82.}
\begin{abstract}
We are interested in modeling networks in which the connectivity among
the nodes and node attributes are random variables and interact with
each other. We propose a probabilistic model that allows one to formulate
jointly a probability distribution for these variables. This model
can be described as a combination of a latent space model and a Gaussian
graphical model: given the node variables, the edges will follow independent
logistic distributions, with the node variables as covariates; given
edges, the node variables will be distributed jointly as multivariate
Gaussian, with their conditional covariance matrix depending on the
graph induced by the edges. We will present some basic properties
of this model, including a connection between this model and a dynamical
network process involving both edges and node variables, the marginal
distribution of the model for edges as a random graph model, its one-edge
conditional distributions, the FKG inequality, and the existence of
a limiting distribution for the edges in a infinite graph.
\end{abstract}

\maketitle

\section{Introduction}

In modeling networks (\cite{Goldenbreg-Zheng-Fienberg-Airoldi,Kolaczyk,Newman}),
the usual focus is on the network topologies, or the configurations
of edges. A network is typically modeled as a random graph (\cite{Bollobas,Chung-Lu,Durrett})
defined in terms of a probability distribution of the edge status.
However, networks in many applied problems are not always just about
links or edges. More extensive data for certain networks, containing
information not only for edges but also for some node variables or
attributes, are becoming available. For such data and for some important
problems in network study, a limitation of the random graph model
is the absence of information from the node variables. Such a model
is incapable of catching interactions between edges and nodes. In
a social network problem, for example, one might be interested in
studying users' behaviors (or some dynamical attribute in the user-profiles)
as a function of the network topology, or vise versa (\cite{McAuley-Leskovec,MVGD}),
or in a gene network problem, one might be interested in inferring
gene expression levels as a function of an underlying regulatory network,
or vise versa (\cite{Wang-Huang}). When it comes to analyzing behaviors
of the nodes in a network or the influence of node behaviors on network
topologies, the utility of the random graph models becomes limited.

The latent space model (\cite{Handcock-Raftery,Hoff-Raftery-Handcock,Kolaczyk})
is another popular network model. It does assume the dependence of
the edge probabilities on some node variables. The model however treats
these variables as latent variables, paying little attention to the
inference on these variables. On the other hand, a Gaussian graphical
model describes a distribution for node variables on a network with
built-in edge information of the network (through the inverse covariance
matrix). It however treats the network topology as a static parameter
which remains constant regardless how the node variables will change.

In this paper, we propose a joint probability distribution for both
edges and node variables (Section 2). A study of such a model can
shed lights on how edges and nodes interact with each other in a network
so that information for both edges and nodes can be utilized in studying
the networks. In a way, our model can be described as a combination
of the latent space model and the Gaussian graphical model: given
the node variables, the edges will follow independent logistic distributions,
with the node variables as covariates in the logistic function; given
edges, the node variables will be distributed jointly as multivariate
Gaussian, with their conditional covariance matrix depending on the
graph induced by the edges. In terms of the marginal distribution
for the node variables, our model generalizes the Gaussian graphical
model to allow for the underlying graphical structure to be random.
In other words, it is now a mixture of Gaussian graphical models over
all the possible edge configurations of the network, and the weights
in this mixture are the probabilities of the corresponding network
configurations. Our model also leads to a non-trivial and interesting
random graphical model when we take the marginal distribution for
the edges in our model. This graphical model is different from all
the other models that we have been in the literature so far. As we
will see, the probability of a network configuration under this random
graphical model is proportional to the square root of the determinant
of the corresponding conditional covariance matrix of the node variables. 

A reason that motivates us to propose such a model is that it provides
a sensible framework for modeling network dynamics in which the edge
status and the node variables change their values over time, as we
will explain in the end of Section 2. We will see that the dynamical
system updates edge status and node variables alternatively according
to the conditional distributions between edges and nodes. The equilibrium
(stable) distribution of the dynamical system is then exactly the
joint distribution we propose here. In other words, our model can
be viewed as the stable probability law of a dynamical network process.
For data containing both edges and node variables, our model can be
fairly easy to fit, because of the simple forms of the conditional
probabilities. If only edge data are available, techniques developed
for the latent space models may be adopted. We will not discuss these
issues any further here. This paper is mainly about basic probabilistic
properties of the model.

We will pay particular attention to the marginal distribution for
edges of our model. An explicit formula for the conditional probability
of one edge given all other edges is given in Section 3. The formula
has a simple form and can be useful for predicting one edge's status
based on observations from other edges. We will show in Section 4
that the probability distribution for edges is positively associated
in the sense that it satisfies the FKG inequality (\cite{Grimmett}),
a property that is shared by many well-known models in statistical
mechanics. We then give a weak convergence result for the edge distribution
based on the FKG inequality. To ensure consistent results in statistical
analysis for very large networks, it is essential that the model,
as a probability law, has a limiting distribution. The concept of
the limit for random graphs we use here is that of the infinite-volume
Gibbs distributions on graphs, involving both nodes and edges (see
\cite{Grimmett} for an example). Our approach does not depend on
the concept of metrics for graphs and therefore is different from
those seen in another line of research on graph limits (\cite{Borgs-Chayes-Lovasz-Vesztergombi-I,Borgs-Chayes-Lovasz-Vesztergombi-II,Diaconis-Janson,Lovasz}).
We also note that there is a close similarity between our model and
the random-cluster model derived in the statistical mechanics (\cite{Grimmett}):
what our model is to the Gaussian graphical model is in some sense
similar to what the random-cluster model is to Ising or Potts models.
This is indeed another reason that motivated us to propose the model
in this note.

\section{The Random Gaussian Graphical Model}

We will call our model the random Gaussian graphical model and formulate
it in this section. Let $G=(V,E)$ be a finite simple graph (undirected,
unweighted, no loops, no multiple edges) with $E$ being a subset
of $V\times V$ which is fixed. Suppose $|V|=m$ and $|E|=n$. For
convenience we identify $V$ as the integer set $V=\{1,...,m\}$.
Suppose associated with each node $i\in V$ there is a random variable
$X_{i}$, representing an attribute of node $i$. Let $X=\{X_{1},...,X_{m}\}$.
We will use $x\in\mathcal{R}^{m}$ to denote a generic value of $X$.
We write $(i,j)$ for the edge in $E$ which is incident with the
nodes $i,j\in V$. 

We will consider random sub-graphs of $G$ in which $V$ remains the
same and $E$ is reduced randomly to some subset of itself. Such a
random graph can be represented by a random adjacency matrix $A=\{A_{ij},\ i,j\in V\}$
in which all the diagonal elements $A_{ii}=0$, and for each edge
$(i,j)\in E$, $A_{ij}=A_{ji}=1$ if the edge is present in the random
graph, and $A_{ij}=A_{ji}=0$ if otherwise. It is always understood
that $A_{ij}\equiv0$ for all $(i,j)\not\in E$. We will call $A_{ij}$
an edge variable. With a slight abuse of notation, we let $\mathcal{A}=\{0,1\}^{E}$
be the set of all possible values of $A$. We will use $a=a^{T}\in\mathcal{A}$
to denote a generic value of the adjacency matrix $A$.

By ``random Gaussian graphical model'' we mean the following joint
probability density for variables $A$ and $X$, defined on the space
$\mathcal{A}\times\mathcal{R}^{m}$, 
\begin{equation}
\mu(a,x)\equiv\frac{1}{Z}\exp\left\{ -\frac{1}{2}H(a,x)\right\} ,\ (a,x)\in\mathcal{A}\times\mathcal{R}^{m},\label{eq:the model}
\end{equation}
where 
\begin{equation}
H(a,x)=\alpha\sum_{i}x_{i}^{2}+\beta\sum_{(i,j)\in E}a_{ij}(x_{i}-x_{j})^{2}\label{eq:H}
\end{equation}
for some parameters $\alpha>0$ and $\beta\ge0$, and 
\[
Z=\sum_{a\in\mathcal{A}}\int_{\mathcal{R}^{m}}\mu(a,x)dx
\]
is the normalizing constant. It is clear that this $Z$ is always
finite. 

We note that in this model, if all $a_{ij}=1$ it becomes an usual
Gaussian graphical model (as we will see below). Therefore we can
consider the Gaussian graphical model as a ``full model'' relative
to the given edge set $E$ while model (\ref{eq:the model}) as a
model that allows us to ``turn off'' some edges in $E$ at random
(and therefore remove the associated correlation terms among the nodes
in (\ref{eq:H})) according to the values of $a_{ij}$'s. Model (\ref{eq:the model})
is like a Gaussian graphical model on a less connected graph obtained
by removing some edges from $E$ randomly, and therefore the name
``random Gaussian graphical model''. In particular, if all $a_{ij}=0$
and therefore there are no connections among the nodes in the graph,
$X_{i}$'s are i.i.d. $N(0,1/\alpha)$ random variables. The joint
distribution of values of $a_{ij}$'s are in turn depended on $X_{i}$'s.
In general, the likelihood of connectivity among the nodes is determined
by the magnitudes of the differences between the corresponding node
variables and the value of $\beta$. On the other hand, the connectivity
of the nodes will, in turn, affect the distribution of the node variables. 

To study this model, it is more convenient to rewrite $\mu(a,x)$
in a matrix form as follows. Let $e_{i}$ be an $m$-dimensional column
vector such that its $i$th element is 1 and all others are 0. We
define a matrix, as a function of $A=a$,
\begin{equation}
Q(a)=\alpha I+\beta\sum_{(i,j)\in E}a_{ij}(e_{i}-e_{j})(e_{i}-e_{j})^{T},\ a\in\mathcal{A}.\label{eq:Q}
\end{equation}
Then
\[
H(a,x)=x^{T}Q(a)x.
\]
Note that the term 
\[
L(a)=\sum_{(i,j)\in E}a_{ij}(e_{i}-e_{j})(e_{i}-e_{j})^{T}
\]
in (\ref{eq:Q}) is the graphical Laplacian of the graph for $A=a$.
Also note that since $H(a,x)>0$ for all $x\not\not\equiv0$ and all
$a\in\mathcal{A}$, $Q(a)$ is positive definite for all $a\in\mathcal{A}$.
Now let 
\[
\Sigma(a)=Q(a)^{-1}.
\]
Then
\begin{eqnarray*}
Z & = & \sum_{a}\int_{\mathcal{R}^{m}}\exp\left\{ -\frac{1}{2}x^{T}Q(a)x\right\} dx\\
 & = & \sum_{a}|2\pi\Sigma(a)|^{\text{1/2}},
\end{eqnarray*}
where $|2\pi\Sigma(a)|$ is the determinant of the matrix $2\pi\Sigma(a)$.
It follows that
\begin{equation}
\mu(a,x)=\frac{|\Sigma(a)|^{\text{1/2}}}{\sum_{a'}|\Sigma(a')|^{\text{1/2}}}\phi(x|0,\Sigma(a)),\label{eq:pi}
\end{equation}
where $\phi(x|0,\Sigma(a))$ is the density function of a 0 mean multivariate
normal distribution with covariance matrix $\Sigma(a)$, $N(0,\Sigma(a))$.

The marginal distribution of (\ref{eq:pi}) for $A$ provides a model
for the random graph and it takes the form:
\begin{equation}
\mu_{A}(a)=\frac{1}{\kappa}|\Sigma(a)|^{1/2},\ a\in\mathcal{A},\label{eq:piA}
\end{equation}
where $\kappa$ is the normalizing constant
\[
\kappa=\sum_{a'}|\Sigma(a')|^{1/2}.
\]
Therefore $\mu_{A}(a)$ is simply proportional to $|\Sigma(a)|^{1/2}$.
The marginal distribution for $X$ can be written as
\begin{equation}
\mu_{X}(x)=\sum_{a}\mu_{A}(a)\phi(x|0,\Sigma(a)),\ x\in\mathcal{R}^{m}.\label{eq:piX}
\end{equation}
This is a mixture of the usual Gaussian graphical models with the
precision matrices $Q(a)$ and the corresponding weight $\mu_{A}(a)$
for each realization of $A$.

Modeling dynamical networks in which connectivity and node variables
are changing in time is an important problem (\cite{Krivitsky-Handcock,Snijder-Van-de-Bunt-Steglich}).
The distribution proposed in (\ref{eq:pi}) relates naturally to such
a dynamical network process which we can easily describe below.

The conditional distributions of (\ref{eq:pi}) have the following
particularly simple forms. Given $X$, $A_{ij}$'s in $A$ are independent
and the corresponding probabilities take the logistic form:
\begin{equation}
\mu(a|x)=\prod_{(i,j)\in E}\left(\frac{1}{1+\exp\left\{ \beta(x_{i}-x_{j})^{2}\right\} }\right)^{a_{ij}}\left(\frac{\exp\left\{ \beta(x_{i}-x_{j})^{2}\right\} }{1+\exp\left\{ \beta(x_{i}-x_{j})^{2}\right\} }\right)^{1-a_{ij}}.\label{eq:a-given-x}
\end{equation}
A latent space model is thus the conditional distribution $\mu(a|x)$
with $x$ being treated as latent variables. Given $A=a$, $X$ is
simply distributed as a multivariate normal: 
\begin{equation}
\mu(x|a)=\phi(x|0,\Sigma(a)).\label{eq:x-given-a}
\end{equation}
This is a centered Gaussian graphical model with the precision matrix
$Q(a)$.

We can think of the joint distribution $\mu(a,x)$ as an equilibrium
or stable state of the following network process driven by the conditional
distributions (\ref{eq:a-given-x}) and (\ref{eq:x-given-a}). Let's
denote the process by $\{A^{(t)},X^{(t)}\}$, $t=0,1,2,...$.

Initially, the network consists of disconnected nodes with i.i.d.
normal random variables in $X^{(0)}=\{X_{1}^{(0)},...,X_{m}^{(0)}\}$:
\[
A^{(0)}\equiv0_{n\times n},\ X^{(0)}\sim N(0,\alpha^{-1}I_{m}).
\]
Suppose at time $t$, the current network is in state $\{A^{(t)},X^{(t)}\}$.
The network updates itself after $t$ at some random time points which
we assume are independent of all $A^{(u)}$ and $X^{(u)}$, $u\le t$.
The connectivity of the network is updated so that, independently,
some $A_{ij}$'s have their status switch between 0 and 1 for some
$(i,j)\in E$ and others remain unchanged, according to the conditional
distribution (\ref{eq:a-given-x}) with the given $X^{(t)}$:
\[
A^{(t+1)}\sim\mu(a|X^{(t)}).
\]
This change of connectivity then modifies the conditional covariances
among the node variables and at some later independent random time
points, values of these variables are updated according to (\ref{eq:x-given-a})
with the updated conditional covariance matrix $\Sigma(A^{(t+1)})$:
\[
X^{(t+1)}\sim N(0,\Sigma(A^{(t+1)})).
\]
The system evolves in time by repeating this updating process.

One can checked that the joint distribution of $A^{(t)}$ and $X^{(t)}$
in this process converges to $\mu$ as $t\to\infty$. In fact, the
dynamical process we just described is exactly the process of a ``Gibbs
sampler'' in MCMC computations. The simple forms of the conditional
probabilities (\ref{eq:a-given-x}) and (\ref{eq:x-given-a}) allow
us to simulate $\mu$ easily via an MCMC procedure.

\section{Conditional Distribution of One Edge Given Others}

We now turn to the marginal distribution of $A$. For notational simplicity
we will often write $\mu(B)$ for $\mu_{A}(B)$ for any edge event
$B$. This section is about dependence in distribution of one edge
on other edges. We will give an explicit formula for this conditional
distribution. It shows in Proposition \ref{prop:conditional dist}
below that this conditional probability depends on other edges only
through the conditional variance of the difference of the corresponding
node variables. This result allows us to show the FKG inequality and
a weak convergence property for our graphical model.

First, some notation. Let us view the network model as a system $\mathcal{S}=(V,E,A,X,\mu)$
with its components defined as in Section 2. Suppose the system $\mathcal{S}$
is such that it does not contain edge $(i',j')$ in $E$. Suppose
the system $\mathcal{S}'=(V,E',A',X,\mu')$ is an augmented version
of $\mathcal{S}$, obtained by adding to $\mathcal{S}$ the edge $(i',j')$.
Note that both systems $S$ and $S'$ have the same node set $V$.
Therefore we have $E'=E\cup\{(i',j')\}$, 
\begin{equation}
A'=A+A_{i'j'}'(e_{i'}e_{j'}^{T}+e_{j'}e_{i'}^{T}),\label{eq:A}
\end{equation}
where $A_{i'j'}'\in\{0,1\}$ is the new edge variable corresponding
to the edge $(i',j')$, and
\begin{equation}
\mathcal{A}'=\{a':a'=a+a_{i'j'}'(e_{i'}e_{j'}^{T}+e_{j'}e_{i'}^{T}),a\in\mathcal{A},a_{i'j'}'\in\{0,1\}\}.\label{eq:a}
\end{equation}
Therefore the only difference between $A$ and $A'$ is that $A_{i'j'}=0$
in $A$ but $A_{i'j'}'$ can be either 0 or 1 in $A'$. This also
applies to the difference between an $a$ and an $a'$. Following
(\ref{eq:Q}), we define the conditional precision matrix for $\mathcal{S}'$
as
\[
Q'(a')=Q(a)+\beta a_{i'j'}'(e_{i'}-e_{j'})(e_{i'}-e_{j'})^{T}.
\]
For simplicity, we will often suppress $a$ and $a'$ in our notation
below when there is no danger of confusion. We will write, for example,
$Q'$ and $\Sigma$ for $Q'(a')$ and $\Sigma(a)$ respectively. Finally
we define $\mu'$ for $\mathcal{S}'$ accordingly based on $Q'$ or,
equivalently, the its inverse $\Sigma'$. It is important to point
out that the marginal distribution for $A$ under $\mu'$ is not the
same as that under $\mu$. In fact $\mu$ and $\mu'$ are related
through the relation

\[
\mu(A=a)=\mu'(A=a|A_{i'j'}'=0).
\]

We first establish a general relationship between $|\Sigma^{'}|$
and $|\Sigma|$ (recall that the probability $\mu(A=a)$ is proportional
to $|\Sigma|^{1/2}$). For all $i,j\in V$, let $\sigma_{ij}$ and
$\sigma_{ij}'$ be the entries of $\Sigma$ and $\Sigma'$ respectively,
the conditional covariance matrices of $X$ under $\mu$ and $\mu'$
respectively, and let 
\[
\delta_{ij}=\sigma_{ii}+\sigma_{jj}-2\sigma_{ij}\ \mbox{and}\ \delta_{ij}'=\sigma_{ii}'+\sigma_{jj}'-2\sigma_{ij}',
\]
the conditional variances of $X_{i}-X_{j}$ under the distributions
$N(0,\Sigma)$ and $N(0,\Sigma')$ respectively. Note that $\delta_{ij}$
and $\delta'_{ij}$ are functions of the network configurations $a$
and $a'$ respectively.
\begin{lem}
\label{lem:lemma}Let $a\in\mathcal{A}$ and $a'\in\mathcal{A}'$
be as related in (\ref{eq:a}). Then
\begin{equation}
|\Sigma'|=\left(1+\beta\delta_{i'j'}\right)^{-a_{i'j'}'}|\Sigma|,\label{eq:d update}
\end{equation}
and $\delta_{i'j'}$ and $\delta'_{i'j'}$ are related through 
\begin{equation}
\left(1-\beta\delta_{i'j'}'\right)^{a_{i'j'}'}\left(1+\beta\delta_{i'j'}\right)^{a_{i'j'}'}=1.\label{eq:dov id}
\end{equation}
\end{lem}
\begin{proof}
By definition 
\[
\Sigma'=(Q+\beta a_{i'j'}'(e_{i'}-e_{j'})(e_{i'}-e_{j'})^{T})^{-1}.
\]
From the Sylvester's identity for determinants 
\[
|I_{n}+UV|=|I_{m}+VU|,
\]
which holds for any $n\times m$ matrix $U$ and $m\times n$ matrix
$V$, it follows that for any $n\times n$ invertible $W$,
\begin{equation}
\left|W+UV\right|=\left|W\right|\cdot\left|I_{n}+(W^{-1}U)V\right|=\left|W\right|\cdot\left|I_{m}+VW^{-1}U\right|.\label{eq:sylvester}
\end{equation}
Now set $W=Q,$ $U=\beta a_{i'j'}(e_{i'}-e_{j'}),$ and $V=(e_{i'}-e_{j'})^{T}.$
Since
\[
VW^{-1}U=\beta a_{i'j'}'(e_{i'}-e_{j'})^{T}\Sigma(e_{i'}-e_{j'})=\beta a_{i'j'}'\delta_{i'j'},
\]
(\ref{eq:d update}) follows from (\ref{eq:sylvester}) with $W+UV=Q^{'}$.
The identity (\ref{eq:dov id}) is based on the observation that $W+UV=Q^{'}$
is nonsingular and
\[
|W|=|(W+UV)-UV|=|W+UV|\times|I-V(W+UV)^{-1}U|
\]
and a comparison of this equation to (\ref{eq:sylvester}).
\end{proof}
To derive a formula for the conditional distribution of one edge given
other edges, we formulate this problem in terms of the systems $\mathcal{S}$
and $\mathcal{S}'$. Let $\delta'_{i'j'}$ be as defined above and
further let 
\[
\delta_{1}(a)\equiv\delta_{i'j'}'(a':a'_{ij}=a_{ij}\mbox{ for \ensuremath{(i,j)\in E} and }a'_{i'j'}=1)
\]
 and 
\[
\delta_{0}(a)\equiv\delta_{i'j'}^{'}(a':a_{ij}'=a_{ij}\mbox{ for \ensuremath{(i,j)\in E} and }a_{i'j'}'=0).
\]
These are the conditional variances under $\mu'$ of the difference
$X_{i}-X_{j}$ for a given $a'$ in which $a_{i'j'}'=1$ and $a_{i'j'}'=0$
respectively. Note that in this notation, $\delta_{0}(a)$ and $\delta_{i'j'}(a)$
have the same value. 
\begin{prop}
\label{prop:conditional dist}The conditional probability of $A_{i'j'}$
given $A=a$ under $\mu'$, is
\begin{equation}
\mu'(A_{i'j'}=1|A=a)=\frac{1}{1+\sqrt{1+\beta\delta_{i'j'}(a)}}\label{eq:conditionP1}
\end{equation}
or, equivalently,
\begin{equation}
\mu'(A_{i'j'}=1|A=a)=\frac{\sqrt{\delta_{1}(a)}}{\sqrt{\delta_{0}(a)}+\sqrt{\delta_{1}(a)}}.\label{eq:conditionP2}
\end{equation}
\end{prop}
\begin{proof}
The first formula follows directly from (\ref{eq:d update}) and (\ref{eq:piA})
by noting that for $a$ and $a'$ as related in (\ref{eq:a}),
\[
\mu'(A_{i'j'}=1|A_{ij}=a_{ij},(i,j)\in E)=\frac{|\Sigma'(a':a_{i'j'}'=1)|^{1/2}}{|\Sigma'(a':a_{i'j'}'=0)|^{1/2}+|\Sigma'(a':a_{i'j'}'=1)|^{1/2}}
\]
and canceling out from both the numerator and the denominator the
common fraction $|\Sigma(a)|^{1/2}$. To obtain the second formula,
we note that with the new notation, (\ref{eq:dov id}) can be written
as
\[
\left(1-\beta\delta_{1}(a)\right)^{a_{i'j'}'}\left(1+\beta\delta_{0}(a)\right)^{a_{i'j'}'}=1
\]
which implies another interesting identity, when $a_{i'j'}'=1$,
\[
\frac{\delta_{0}(a)}{\delta_{1}(a)}=1+\beta\delta_{0}(a)=1+\beta\delta_{i'j'}(a).
\]
Plugging this into (\ref{eq:conditionP1}), we obtain (\ref{eq:conditionP2}).
\end{proof}
This proposition asserts that, among other things, the conditional
distribution of an edge $A_{ij}$ in the graph depends on the rest
of the edges through and only through the conditional standard deviations
of $X_{i}-X_{j}$.

\section{The FKG Inequality and the Infinite-Volume Weak Limit}

In this section we study an asymptotic property of the edge distribution
as the size of the graph approaches to infinity. We will first explore
a property called positive association for our model which, together
with the Proposition \ref{prop:conditional dist} in Section 3, will
allow us to show the existence of a weak limit (convergence in distribution)
on an infinite graph for this model. After Lemma \ref{prop:update Sigma}
below, the properties for our graphical model are developed in parallel
to some of those for the random-cluster model of \cite{Grimmett}.

We start with a relationship between the conditional covariance matrices
$\Sigma'$ and $\Sigma$ given edges. It displays explicitly how adding
one edge to the graph can affect the conditional covariances of the
random variables in $X$.
\begin{lem}
\label{prop:update Sigma}For all $a$ and $a'$ as related in (\ref{eq:a})
and for any $i,j\in V$,
\begin{equation}
\sigma_{ij}'=\sigma_{ij}-\frac{\beta a_{i'j'}'}{1+\beta\delta_{i'j'}}(\sigma_{ii'}-\sigma_{ij'})(\sigma_{ji'}-\sigma_{jj'})\label{eq:sigma update}
\end{equation}
and
\begin{equation}
\delta_{ij}'=\delta_{ij}-\frac{\beta a_{i'j'}'}{1+\beta\delta_{i'j'}}(\sigma_{ii'}-\sigma_{ij'}-\sigma_{ji'}+\sigma_{jj'})^{2}.\label{eq:vod update}
\end{equation}
\end{lem}
\begin{proof}
Applying the matrix identity
\[
(W+UV)^{-1}=W^{-1}-W^{-1}V(I+UW^{-1}V)^{-1}UW^{-1},
\]
which holds for any valid (so that the inverses exist) and compatible
(so that the products are defined) matrices, we have, as in the proof
of the Proposition \ref{lem:lemma},
\[
\Sigma'=\Sigma-\frac{\beta a_{i'j'}'}{1+\beta(e_{i'}-e_{j'})^{T}\Sigma(e_{i'}-e_{j'})}\Sigma(e_{i'}-e_{j'})(e_{i'}-e_{j'})^{T}\Sigma.
\]
Now (\ref{eq:sigma update}) is just an element-wise version of this
equation, and (\ref{eq:vod update}) is a direct consequence of (\ref{eq:sigma update})
and the definitions of $\delta_{ij}$ and $\delta'_{ij}$.
\end{proof}
A special case of (\ref{eq:sigma update}) is $i=j$. It implies that
the conditional variance of $X_{i}$ for each $i\in V$ become smaller
when extra edges are added into the conditioning network or, in a
more general form, 
\begin{equation}
Var(X_{i}|a)\ge Var'(X_{i}|a'),\ \forall i\in V,\label{eq:Var}
\end{equation}
whenever $a\le a'$ in the sense that $a_{ij}\le a'_{ij}$ for all
$i,j\in V$. Similarly, (\ref{eq:vod update}) shows the conditional
variance of $X_{i}-X_{j}$ is also decreasing in $a$:
\begin{equation}
Var(X_{i}-X_{j}|a)\ge Var'(X_{i}-X_{j}|a').\label{eq:deltaVar}
\end{equation}
Hence a network with more edges has smaller conditional variances
of $X_{i}$ and $X_{i}-X_{j}$, for all $i,j\in V$. 

We now show that our graphical model possesses a nice property called
positive association which is characterized by the FKG inequality
below (\cite{Grimmett}). This inequality plays a fundamental role
in studying some well-known models in statistical mechanics, including
the Ising model and Potts model.
\begin{prop}
For any increasing functions $f$ and $g$ defined on $\{0,1\}^{E}$,
we have
\begin{equation}
E_{\mu}(fg)\ge E_{\mu}(f)E_{\mu}(g).\label{eq:FKG}
\end{equation}
\end{prop}
\begin{proof}
For any $a,a'\in\{0,1\}^{E}$, let 
\[
a\vee a'=\{\max\{a_{ij},a_{ij}'\}\mbox{ for all }(i,j)\in E\mbox{ and }a_{ij}\in a,a'_{ij}\in a'\}
\]
and 
\[
a\wedge a'=\{\min\{a_{ij},a_{ij}'\}\mbox{ for all }(i,j)\in E\mbox{ and }a_{ij}\in a,a'_{ij}\in a'\}.
\]
It is well-known that (\ref{eq:FKG}) is a consequence of the FKG
lattice condition
\begin{equation}
\mu(a\vee a')\mu(a\wedge a')\ge\mu(a)\mu(a').\label{eq:FKG condition}
\end{equation}
A statement and a proof of this result can be found in \cite{Grimmett},
page 25-26. According to again \cite{Grimmett} (Theorem 2.24), the
condition (\ref{eq:FKG condition}) is in turn equivalent to the ``one-point
conditional probability condition''. This later condition states
that for any $(i,j)\in E$ and a given $a\in\{0,1\}^{E-\{(i,j)\}}$,
if $A^{-ij}\equiv\{A_{kl}:(k,l)\in E-\{(i,j)\}\}$, the conditional
probability $\mu(A_{ij}=1|A^{-ij}=a)$ is increasing in $a$. To show
this is true in our case, we note that according to (\ref{eq:conditionP1})
of Proposition \ref{prop:conditional dist}, such a conditional probability
is decreasing in the quantity $\delta_{ij}(a)$. Also (\ref{eq:vod update})
of Lemma \ref{prop:update Sigma} implies that for every $(i,j)\in E$,
the function $a\to\delta_{ij}(a)$ is decreasing:
\[
\delta_{ij}(a')\ge\delta_{ij}(a''),\mbox{ whenever }a'\le a''.
\]
It follows that the one-point conditional probability condition holds,
and therefore (\ref{eq:FKG}) is true.
\end{proof}
Finally we consider any sequence of finite graphs $G^{(n)}=(V^{(n)},E^{(n)})$,
$n=1,2,...$, such that $V^{(n)}\subset V^{(n+1)}$, $E^{(n)}\subset E^{(n+1)}$,
and the size of $E^{(n)}$ tends to infinity. Let $A^{(n)}$ be the
random adjacency matrix defined on $G^{(n)}$ with the probability
distribution $\mu_{n}\equiv\mu_{A^{(n)}}$ as given in (\ref{eq:piA}):
\begin{equation}
\mu_{n}(a)=\frac{|\Sigma(a)|^{1/2}}{\sum_{a'\in\{0,1\}^{E^{(n)}}}|\Sigma(a')|^{1/2}},\ a\in\{0,1\}^{E^{(n)}}.\label{eq:mu_n}
\end{equation}
We are interested in the consistency property of the sequence of the
distributions $\{\mu_{n}\}$ so that for any event $B$ depending
on some edges in a finite graph $G^{(n_{0})}$, the probabilities
$\mu_{n}(B)$, $n\ge n_{0}$, has a limit as $n\to\infty$. This problem
can be formulated as follows. 

Let 
\begin{equation}
V=\lim V^{(n)}\mbox{ and }E=\lim E^{(n)},\label{eq:V_and_E}
\end{equation}
the limits of the increasing sets. Let $\Omega=\{0,1\}^{E}$ and $\mathcal{F}$
be the $\sigma$-field generated by the cylinder events of $\Omega$.
For each $n$, we can view $\mu_{n}$ as a probability measure defined
on $(\Omega,\mathcal{F})$ with a support on a subset 
\[
\Omega_{n}=\{0,1\}^{E^{(n)}}\times\{0\}^{E-E^{(n)}}
\]
of $\Omega$. The consistency problem is then the problem of weak
convergence of the sequence $\{\mu_{n}\}$ on $(\Omega,\mathcal{F})$.

Let $\mathcal{F}_{n}$ be the $\sigma$-field generated by the subsets
of $\Omega_{n}$. Then $\mathcal{F}_{n}\subset\mathcal{F}_{n+1}$.
We notice that for any $m<n$ and $B_{1}\in\mathcal{F}_{m}$, if $B_{2}\in\mathcal{F}_{n}$
is the event that $A_{ij}=0$ for all edges in $E^{(n)}-E^{(m)}$:
$B_{2}=(A_{ij}=0,\forall(i,j)\in E^{(n)}-E^{(m)})$, then according
to (\ref{eq:mu_n}) and (\ref{eq:Q}), 
\begin{equation}
\mu_{m}(B_{1})=\mu_{n}(B_{1}|B_{2}).\label{eq:mu_m_and_mu_n}
\end{equation}
We will say an edge event $B$ is increasing if the corresponding
indicator function $I_{B}(a)$ is increasing. If $B$ is decreasing
then the negative of its indicator function is increasing.
\begin{cor}
\label{cor:monotone}For any $m<n$ and any increasing event $B_{1}\in\mathcal{F}_{m}$,
\begin{equation}
\mu_{m}(B_{1})\le\mu_{n}(B_{1}).\label{eq:monotone}
\end{equation}
\end{cor}
\begin{proof}
Let $B_{2}$ be as defined above. Then $B_{2}$ is a decreasing event
and $\mu_{n}(B_{2})>0$. The FKG inequality (\ref{eq:FKG}) in this
case implies
\[
\mu_{n}(B_{1}\cap B_{2})\le\mu_{n}(B_{1})\mu_{n}(B_{2}),
\]
or
\[
\mu_{n}(B_{1}|B_{2})\le\mu_{n}(B_{1}).
\]
(\ref{eq:monotone}) follows from this and (\ref{eq:mu_m_and_mu_n}).
\end{proof}
Now we have:
\begin{prop}
\label{prop:Limit}As $n\to\infty$, $A^{(n)}$ converges in distribution
to a random adjacency matrix $A$ of the graph $G$ in the probability
space $(\Omega,\mathcal{F},\nu)$ for some probability distribution
$\nu$.\end{prop}
\begin{proof}
The statement in the proposition is equivalent to asserting the weak
convergence of $\mu_{n}$ to some $\nu$ on $(\Omega,\mathcal{F})$
or, equivalently, to the statement that there is a $\nu$ on $(\Omega,\mathcal{F})$
such that for every finite dimensional cylinder event $B\in\mathcal{F}_{n}\subset\mathcal{F}$,
\begin{equation}
\lim_{n}\mu_{n}(B)=\nu(B).\label{eq:limit}
\end{equation}
Note that $B$ is discrete with $\partial B=\emptyset$ in the discrete
topology and therefore is always $\nu$-continuous ($\nu(\partial B)=0$).
To establish (\ref{eq:limit}), we follow an argument of Grimmett
in proving his Theorem 4.19 (a) in \cite{Grimmett} by first assuming
that $B$ is an increasing event. The Corollary \ref{cor:monotone}
then states that
\[
\mu_{n}(B)\le\mu_{n+1}(B).
\]
Therefore the limit (\ref{eq:limit}) holds for all increasing events
$B$. Since the set of increasing events forms a convergence-determining
class (Billingsley \cite{Billingsley}), (\ref{eq:limit}) must hold
for all the subsets in $\mathcal{F}$ for a probability measure $\nu$
on $(\Omega,\mathcal{F})$.
\end{proof}
We end this section with two more observations. 

First, suppose $\mathcal{S}'=(V',E',A',X',\mu')$ and $\mathcal{S}''=(V'',E'',A'',X'',\mu'')$
be two finite systems defined as before such that $V'\subset V''$,
$E'\subset E''$. Because of the monotonic properties of the functions
$a\to Var(X_{i}|a)$ and $a\to\delta_{ij}(a)$ implied in (\ref{eq:Var})
and (\ref{eq:deltaVar}), Corollary \ref{cor:monotone} implies that
for all $i,j\in V$, the variances of $X_{i}$ and $X_{i}-X_{j}$,
as functions of the edge set $E$, are decreasing in the sense that
\[
Var'(X_{i})\ge Var''(X_{i}),\ Var'(X_{i}-X_{j})\ge Var''(X_{i}-X_{j}).
\]

Finally, the formula (\ref{eq:sigma update}) demonstrates exactly
how $\Sigma'$ depends on $A_{i'j'}$. This observation leads to a
martingale representation for the quantity $\log|\Sigma(A)|$ as follows.
To state the result, let us label all edges in the graph in some (arbitrary)
order so that we can write $E=\{(i_{1},j_{1}),\cdots,(i_{k},j_{k})\}$.
\begin{prop}
Let $A=\{A_{ij}\}_{i,j\in V}$ be any symmetric random adjacency matrix
from the distribution $\mu_{A}$. For $k=1,...,n$, let 
\[
A^{(k)}=\sum_{l=1}^{k}A_{i_{l}j_{l}}(e_{i_{l}}-e_{j_{l}})(e_{i_{l}}-e_{j_{l}})^{T}
\]
and $\mathcal{F}_{k}=\sigma\{A_{i_{1}j_{1}},...,A_{i_{k}j_{k}}\}$.
Then $\{-\log|\Sigma(A^{(k)})|,\mathcal{F}_{k},k=1,...,n\}$ forms
a sub-martingale.\end{prop}
\begin{proof}
Applying (\ref{eq:d update}) repeatedly, we can write, for $k=1,...,n$,
\[
-\log\left|\Sigma(A^{(k)})\right|=\sum_{l=1}^{k}A_{i_{l}j_{l}}\log\left[1+\beta\delta_{i_{l}j_{l}}(A^{(l-1)})\right],
\]
where $A^{(0)}$ is defined to be the 0 matrix. The Proposition follows
by noting that $\delta_{i_{l}j_{l}}(A^{(l-1)})\in\mathcal{F}_{l-1}$
for $l=1,...,n$ (with $\mathcal{F}_{0}$ being the trivial $\sigma$-field)
and all the terms in the summation above are positive.\end{proof}

\end{document}